\documentclass[11pt,reqno]{amsart}

\usepackage[a4paper,margin=1.15in]{geometry}
\usepackage{amsmath,amssymb,amsthm,mathtools}
\usepackage{xcolor}
\usepackage{microtype}
\usepackage{enumitem}
\usepackage{hyperref}

\hypersetup{
  colorlinks=true,
  linkcolor=blue!50!black,
  citecolor=blue!50!black,
  urlcolor=blue!50!black
}

\numberwithin{equation}{section}

\newtheorem{theorem}{Theorem}[section]
\newtheorem{lemma}[theorem]{Lemma}
\newtheorem{proposition}[theorem]{Proposition}
\theoremstyle{remark}

\DeclareMathOperator{\vol}{vol}
\newcommand{\R}{\mathbb{R}}
\newcommand{\G}{G}

\title[On the monotonicity of affine quermassintegrals]
{On the monotonicity of affine quermassintegrals}

\author{Shibing Chen}
\address{School of Mathematical Sciences,
University of Science and Technology of China,
Hefei, Anhui 230026, China}
\email{chenshib@ustc.edu.cn}

\author{Yuanyuan Li}
\address{Institute for Theoretical Sciences,
Westlake University, Hangzhou, 310030, China}
\email{lyyuan@westlake.edu.cn}

\author{Xianduo Wang}
\address{School of Mathematical Sciences,
University of Science and Technology of China,
Hefei, Anhui 230026, China}
\email{xdwang@ustc.edu.cn}
\date{}

\begin{document}

\begin{abstract}
Lutwak's affine quermassintegral theory is a foundational component of
modern affine Brunn--Minkowski theory. Developed in the 1980s, it provides
affine analogues of the classical quermassintegrals and has led to a rich
family of sharp affine isoperimetric inequalities. A central question in
this program, going back to Lutwak's 1988 work, is an
Alexandrov--Fenchel-type monotonicity principle for the normalized
\(L^{-n}\)-moment quermassintegrals \(I_{k,-n}\). In one form, this
principle predicts that
\[
        I_{m,-n}(K)^{1/m}\ge I_{k,-n}(K)^{1/k},
        \qquad 1\le m<k\le n .
\]
The question was recorded in Gardner's 2006 book
\emph{Geometric Tomography} as part of its problem list, and the comparison
with the top dimension, \(k=n\), was established by Milman and Yehudayoff
in their 2023 JAMS paper.

We show that the proposed monotonicity does not persist in the full range.
More precisely, for every triple of integers \(m,k,n\) satisfying
\(1\le m<k\le n-1\) and \(n>(m+2)(k+2)-2\), there exists an
origin-symmetric \(C^2_+\) convex body \(K\subset\mathbb R^n\) such that
\[
        I_{m,-n}(K)^{1/m} < I_{k,-n}(K)^{1/k}.
\]
The example is obtained from the Euclidean ball by an arbitrarily small
degree-four spherical harmonic perturbation.

On the positive side, we prove that the endpoint chain is true in dimension
three: for every convex body \(K\subset\mathbb R^3\),
\[
        I_{1,-3}(K)\ge I_{2,-3}(K)^{1/2}\ge I_{3,-3}(K)^{1/3}=1.
\]
The equality cases in both non-trivial inequalities are exactly ellipsoids,
up to translation and nonsingular affine transformations.
\end{abstract}

\maketitle

\section{Introduction}

Let $K\subset\R^n$ be a convex body, let $B_K$ denote the centered Euclidean ball with $\vol_n(B_K)=\vol_n(K)$, and let $\G_{n,k}$ be the Grassmannian of $k$-dimensional linear subspaces of $\R^n$ equipped with its normalized rotation-invariant probability measure. For $p\ne 0$, Milman and Yehudayoff define the normalized $L^p$-moment quermassintegral
\begin{equation}\label{eq:I-def}
    I_{k,p}(K)
    =
    \left(
      \frac{\displaystyle\int_{\G_{n,k}} \vol_k(P_FK)^p\,dF}
           {\displaystyle\int_{\G_{n,k}} \vol_k(P_FB_K)^p\,dF}
    \right)^{1/p},
\end{equation}
where $P_F$ denotes orthogonal projection onto $F\in \G_{n,k}$.

The problem considered here originates in Lutwak's affine quermassintegral
program from the 1980s.  In the classical theory, the quermassintegrals, or intrinsic volumes, are obtained from Steiner's formula, and Kubota's formula expresses them as averages of projection volumes.  Together with the Alexandrov--Fenchel inequalities, this gives the standard monotone chain in the Euclidean setting; see \cite{BonnesenFenchel,Schneider}.  But these classical quantities are not invariant under general volume-preserving affine maps.  Lutwak's idea was to replace the usual $L^1$ average of projection volumes by the critical negative average, the $L^{-n}$ average, and so to define affine quermassintegrals \cite{Lutwak1984,Lutwak1993}.  Grinberg proved the affine invariance of these quantities \cite{Grinberg1991}.  The endpoint question $p=-n$, which is the one used in this note, was already conjectured in Lutwak's 1988 paper \cite{Lutwak1988}; it is also listed in Gardner's book \cite[Problem 9.5]{Gardner}.   The cases $k=1$ and $k=n-1$ are respectively the Blaschke--Santal\'o inequality and Petty's projection inequality \cite{Blaschke1923,Santalo1949,MeyerPajor1990,LutwakZhang1997,Petty1971,Petty1985,LutwakYangZhang2000}.  Milman and Yehudayoff later gave a broader $L^p$-moment formulation, proved the sharp affine quermassintegral minimization theorem, and stated the full monotonicity chain as Conjecture 1.6 \cite{MilmanYehudayoff}.

The conjectural affine Alexandrov--Fenchel chain is
\begin{equation}\label{eq:conj-chain}
    I_{1,p}(K)\ge I_{2,p}(K)^{1/2}\ge \cdots \ge
    I_{k,p}(K)^{1/k}\ge\cdots\ge I_{n,p}(K)^{1/n}=1,
    \qquad -n\le p< 0.
\end{equation}
This is precisely Conjecture 1.6 of Milman and Yehudayoff \cite{MilmanYehudayoff}.  They prove \eqref{eq:conj-chain} in the range
\[
        I_{\ell,p}(K)^{1/\ell}\ge I_{r,p}(K)^{1/r}
        \qquad\text{whenever}\qquad r\ge -p,
\]
and identify the remaining range $1\le \ell<r<-p$ as the open part of the problem.

In this paper we identify a local obstruction to the conjectured monotonicity.
It occurs at the endpoint \(p=-n\), in the comparison between arbitrary fixed
projection dimensions \(m<k\), whenever \(n>(m+2)(k+2)-2\).

The proof is local around the Euclidean ball. We take a small support-function perturbation $h_t=1+tY$, where $Y$ is an even spherical harmonic of degree four. Let \(K_t\) be the convex body with support function \(h_t\). For each subspace $F\in\G_{n,j}$, the projection $P_FK_t$ has support function $h_t|_{\mathbb{S}^{n-1}\cap F}$, so the standard surface-area formula gives a uniform second-order expansion of $\vol_j(P_FK_t)$. After the global volume normalization, we expand the negative moment $\int_{\G_{n,j}}\left(\frac{V_j(F,t)}{V_n(t)^{j/n}}\right)^{-n}\,dF$ and hence $\log I_{j,-n}(K_t)^{1/j}$. The second variation separates into two pieces: a universal Bochner/volume term depending only on the spherical Laplace eigenvalue of $Y$, and a genuinely Grassmannian term involving the square of the projection Radon average $R_jY$. The argument uses this projection-volume second variation to test the
cross-dimensional affine inequality directly. The remaining part is an
explicit computation of the relevant Grassmannian averages for the fourth
zonal harmonic in arbitrary projection dimension. This yields the coefficient of \(\|Y\|_2^2t^2\)
\[
    \frac{3(n+4)(k-m)\big((m+2)(k+2)-n-2\big)}
         {2(n-1)(m+2)(k+2)},
\]
which is negative when \(n>(m+2)(k+2)-2\) and therefore gives the desired strict
reverse inequality.

\begin{theorem}\label{thm:main}
For every triple of integers $m,k,n$ satisfying $1\leq m<k\le n-1$ and
$n>(m+2)(k+2)-2$, there exists an origin-symmetric $C^2_+$ convex body
$K\subset\R^n$ such that
\begin{equation}\label{eq:main-ineq}
        I_{m,-n}(K)^{1/m} < I_{k,-n}(K)^{1/k}.
\end{equation}
Consequently, the chain \eqref{eq:conj-chain} does not hold in full generality at \(p=-n\).
\end{theorem}

The high-dimensional obstruction leaves open the low-dimensional endpoint.  In
three dimensions the full endpoint chain is in fact true.

\begin{theorem}\label{thm:dim3-chain}
For every convex body \(K\subset\R^3\) with non-empty interior,
\begin{equation}\label{eq:dim3-chain-intro}
        I_{1,-3}(K)\ge I_{2,-3}(K)^{1/2}\ge I_{3,-3}(K)^{1/3}=1.
\end{equation}
Equality in either of the two non-trivial inequalities occurs if and only if
\(K\) is an ellipsoid, up to translation and a non-singular affine map.
\end{theorem}
The proof of the three-dimensional endpoint chain is rather different from
the local perturbation argument used for the counterexamples.  The main point
is that, in dimension three, the first comparison can be reduced to a planar
inequality after taking the difference body and then passing to the polar
body.  More precisely, the projections of the difference body are identified
with the polars of central sections of its polar body.  We then use a sharp
two-dimensional determinant estimate, whose proof relies on centroid bodies,
the special planar identity between projection bodies and rotated bodies,
Minkowski's inequality, and the Blaschke--Santaló inequality.  Finally, the
Blaschke--Petkantschin formula turns this estimate on planar sections into
the desired three-dimensional inequality.  This is the main reason why the
three-dimensional case can be proved by a global argument, rather than by the
local second-variation method.

We now describe the local construction used in the proof of
Theorem~\ref{thm:main}, starting with the choice of the perturbation.  Since the quantities
\(I_{j,-n}\) are invariant under volume-preserving affine transformations,
the second variation cannot detect a genuine instability in the quadratic
harmonic directions: such directions are precisely the infinitesimal
deformations of the ball into ellipsoids, and along ellipsoids the affine
normalization makes the relevant quantities constant.  The constant mode is
removed by the volume normalization, while odd harmonics would destroy the
origin-symmetry that we keep throughout the construction.  Thus the first
possible even perturbations which are not generated by affine deformations
occur in degree \(4\).  For this reason we take a fourth-order spherical
harmonic.  The particular zonal choice
\begin{equation}\label{eq:Y-def}
    Y(u)
    = u_1^4-\frac{6}{n+4}u_1^2+\frac{3}{(n+2)(n+4)},
    \qquad u\in \mathbb{S}^{n-1}
\end{equation}
is only for convenience: it is the simplest degree-four harmonic depending
on one coordinate, and it reduces all Grassmannian averages to elementary
one-variable moment computations.  In this sense the instability found below
is the first non-affine, origin-symmetric direction visible in the local
spectral expansion at the Euclidean ball.
For sufficiently small $t\in\R$, the function
\begin{equation}\label{eq:h-def}
        h_t(u)=1+tY(u),\qquad u\in \mathbb{S}^{n-1},
\end{equation}
is the support function of an origin-symmetric $C^2_+$ convex body $K_t$. For the fixed triple $m,k,n$ in Theorem \ref{thm:main}, we prove that
\begin{equation}\label{eq:second-variation-final}
    \log \frac{I_{m,-n}(K_t)^{1/m}}{I_{k,-n}(K_t)^{1/k}}
    =
    \frac{3(n+4)(k-m)\big((m+2)(k+2)-n-2\big)}
         {2(n-1)(m+2)(k+2)}\,
      \|Y\|_{2}^2\,t^2
      +O(t^3).
\end{equation}
The coefficient of $t^2$ is negative when $n>(m+2)(k+2)-2$, which proves the theorem.

The remainder of the paper is organized as follows. 
In Section \ref{pre}, we introduce the notation and collect the basic facts used throughout the paper. 
In Section \ref{sec:second-variation}, we derive the second-variation formula for the normalized negative projection-volume moments under small support-function perturbations of the Euclidean ball. 
In Section \ref{sec:fourth-harmonic}, we specialize to a fourth zonal spherical harmonic and compute the Grassmannian averages needed for arbitrary projection dimensions. 
In Section \ref{sec:proof}, we combine these computations to prove Theorem \ref{thm:main}; the resulting second-order coefficient is negative when \(n>(m+2)(k+2)-2\), which proves the asserted reverse inequality.
In Section \ref{sec:dim3-chain}, we prove Theorem \ref{thm:dim3-chain}.  In the final section, we also show that the three-dimensional argument can be adapted to prove the first endpoint comparison in dimension four.
\section{Preliminaries}\label{pre}

All measures on spheres and Grassmannians are normalized to have total mass one; all integrals over these spaces use the normalized measures unless an unnormalized surface element $d\omega$ is explicitly displayed. We write $\kappa_j=\vol_j(B_2^j)$ where $B^j_2$ denotes the unit ball in $\mathbb{R}^j$ with respect to the $\ell^2$-norm, and
\[
    S_F=\mathbb{S}^{n-1}\cap F,\qquad F\in\G_{n,j}.
\]
By $C^2_+$ we mean that the support function is $C^2$ and that the curvature matrix $\nabla^2h+hg$ is positive definite.
We fix $e_1=(1,0,\ldots,0)\in\R^n$ and write $y_1=\langle y,e_1\rangle$ for the first coordinate of a vector $y\in\R^n$. Thus $u_1$ is the first coordinate of the integration variable $u\in \mathbb{S}^{n-1}$, and $x_1$ is the first coordinate of the integration variable $x\in S_F$.
For an integrable function $f$ on $\mathbb{S}^{n-1}$ define its spherical Radon average over $F$ by
\begin{equation}\label{eq:Rj}
    R_jf(F)=\int_{S_F} f(u)\,d\sigma_F(u),
\end{equation}
where $\sigma_F$ is the normalized measure on $S_F$.  When \(j=1\), \(S_F\)
consists of two antipodal points and all expressions involving the intrinsic
gradient on \(S_F\) are interpreted as zero.  If $K$ has support function
$h_K$, then $P_FK$ has support function $h_K|_{S_F}$.

For $j\ge2$, if $h$ is a $C^2$ support function on $\mathbb{S}^{j-1}$ whose curvature matrix $\nabla^2h+h g$ is positive definite, then the volume of the corresponding $j$-dimensional convex body is
\begin{equation}\label{eq:support-volume}
    \vol_j(L)=\frac1j\int_{\mathbb{S}^{j-1}} h\det(\nabla^2 h+h g)\,d\omega.
\end{equation}
Here $g$ is the round metric, $\nabla$ is its Levi-Civita connection, and $d\omega$ is the usual surface measure. Formula \eqref{eq:support-volume} is the standard support-function expression for mixed volume with all arguments equal to $L$; for the perturbative calculation below, it may also be taken as a direct local parametrization formula for the boundary.

We use the sign convention
\[
        -\Delta_{\mathbb{S}^{n-1}}Y=\lambda Y,
        \qquad \int_{\mathbb{S}^{n-1}}Y^2\,d\sigma=\|Y\|_2^2.
\]

All Laplacians below are Laplace--Beltrami operators with the convention
$
\Delta=\operatorname{div}\nabla
$, so that the spectrum on the unit sphere is non-positive.

We shall use the following integrated Bochner identity on the unit sphere. It is a
standard consequence of the Bochner--Weitzenbock formula; we include the short
proof in order to make the perturbation calculation self-contained. See, for
example, \cite[Chapter 7]{Petersen} for the general Riemannian identity.

\begin{lemma}[Integrated Bochner identity on the sphere]\label{lem:bochner}
Let $f\in C^2(\mathbb{S}^{d})$, $d\ge 1$. Then
\begin{equation}\label{eq:bochner-integrated}
    \int_{\mathbb{S}^{d}}\big((\Delta f)^2-|\nabla^2f|^2\big)\,d\omega
    =(d-1)\int_{\mathbb{S}^{d}}|\nabla f|^2\,d\omega .
\end{equation}
Here $\nabla^2f$ is the covariant Hessian and all norms are taken with respect to
the round metric on $\mathbb{S}^{d}$.
\end{lemma}

\begin{proof}
It is enough to prove the identity for smooth $f$; the stated $C^2$ case follows by approximation in $C^2$ on the compact sphere. The pointwise Bochner formula is
\[
    \frac12\Delta |\nabla f|^2
    =\langle\nabla f,\nabla\Delta f\rangle
      +|\nabla^2 f|^2+\operatorname{Ric}(\nabla f,\nabla f).
\]
On the unit sphere $\mathbb{S}^{d}$, $\operatorname{Ric}=(d-1)g$. Integrating over the
closed manifold $\mathbb{S}^{d}$ eliminates the left-hand side. Moreover,
\[
    \int_{\mathbb{S}^{d}}\langle\nabla f,\nabla\Delta f\rangle\,d\omega
    =-\int_{\mathbb{S}^{d}}(\Delta f)^2\,d\omega
\]
by integration by parts. Therefore
\[
    0=-\int_{\mathbb{S}^{d}}(\Delta f)^2\,d\omega
      +\int_{\mathbb{S}^{d}}|\nabla^2 f|^2\,d\omega
      +(d-1)\int_{\mathbb{S}^{d}}|\nabla f|^2\,d\omega,
\]
which is \eqref{eq:bochner-integrated}.
\end{proof}

\section{A second-variation formula}
\label{sec:second-variation}

We first fix the meaning of the error notation used in this section. If
$Q(F,t)$ is defined for $F\in\G_{n,j}$ and $|t|$ small, then
\[
        Q(F,t)=O(t^3) \quad\text{uniformly in }F
\]
means that there exist constants $\varepsilon>0$ and $C<\infty$, depending only
on the fixed dimension, on $j$, and on the perturbing function through a fixed
$C^2$ bound, such that
\[
        \sup_{F\in\G_{n,j}} |Q(F,t)|\le C|t|^3,
        \qquad |t|\le\varepsilon .
\]
For scalar quantities with no Grassmannian parameter, such as the final
expansion of $\log \frac{I_{m,-n}(K_t)^{1/m}}{I_{k,-n}(K_t)^{1/k}}$, the same notation means the
corresponding one-variable bound $|Q(t)|\le C|t|^3$ for all sufficiently small
$|t|$. In particular, once the quadratic coefficient is nonzero, the remainder
cannot affect its sign for all sufficiently small nonzero $t$.

The next lemma records the projection-volume expansion used throughout the proof. The term $|\nabla_F f|$ denotes the intrinsic spherical gradient of $f|_{S_F}$ on $S_F$. With the convention made above for \(j=1\), the second-order term is zero; indeed the projection is then an interval and the formula is immediate.

\begin{lemma}\label{lem:projection-expansion}
Let $f\in C^2(\mathbb{S}^{n-1})$ and let $K_t$ be the body with support function $h_t=1+tf$, for $|t|$ sufficiently small. Then, uniformly in $F\in\G_{n,j}$,
\begin{equation}\label{eq:projection-expansion}
    \frac{\vol_j(P_FK_t)}{\kappa_j}
    =
    1+jR_jf(F)t
    +\frac{j}{2}R_j\big((j-1)f^2-|\nabla_Ff|^2\big)(F)t^2
    +O(t^3).
\end{equation}
\end{lemma}

\begin{proof}
When $j=1$, write $F=\operatorname{span}\{u\}$. Then $P_FK_t$ is the interval of length
$h_t(u)+h_t(-u)=2+t(f(u)+f(-u))$. Since $\kappa_1=2$ and
$R_1f(F)=\frac12(f(u)+f(-u))$, \eqref{eq:projection-expansion} holds with zero second-order term. Assume now $j\ge2$.

It is enough to prove the formula on a fixed $j$-dimensional Euclidean space.
Write $d=j-1$, identify the unit sphere in that space with $\mathbb{S}^{d}$, and put
$A=\nabla^2f+fg$. Then
\[
    \operatorname{tr}A=\Delta f+d f,
    \qquad
    |A|^2=|\nabla^2 f|^2+2f\Delta f+d f^2.
\]
The determinant expansion
\[
    \det(I+tA)
    =1+t\operatorname{tr}A+\frac{t^2}{2}
      \big((\operatorname{tr}A)^2-|A|^2\big)+O(t^3)
\]
and the support-function formula \eqref{eq:support-volume} yield
\begin{align*}
\frac{\vol_j(L_t)}{\kappa_j}
&=1+j\bar f\,t \\
&\quad +\frac1{\kappa_j j}\int_{\mathbb{S}^{d}}
\left\{
 f(\Delta f+df)
+\frac12\left((\Delta f+df)^2-|\nabla^2f+fg|^2\right)
\right\}
d\omega\,t^2
+O(t^3),
\end{align*}
where $\bar f$ is the normalized spherical average on $\mathbb{S}^d$. Expanding the
integrand in the second-order term gives
\[
\frac12\big((\Delta f)^2-|\nabla^2f|^2\big)
+d f\Delta f+\frac{d(d+1)}2 f^2.
\]
By integration by parts,
\[
    \int_{\mathbb{S}^d} f\Delta f\,d\omega=-\int_{\mathbb{S}^d}|\nabla f|^2\,d\omega,
\]
and by Lemma \ref{lem:bochner},
\[
    \int_{\mathbb{S}^d}\big((\Delta f)^2-|\nabla^2f|^2\big)\,d\omega
    =(d-1)\int_{\mathbb{S}^d}|\nabla f|^2\,d\omega.
\]
Thus the second-order integral equals
\[
    \frac{d+1}{2}\int_{\mathbb{S}^d}\big(d f^2-|\nabla f|^2\big)\,d\omega.
\]
Since $d+1=j$ and $|\mathbb{S}^d|=j\kappa_j$, after division by $j\kappa_j$ this becomes
\[
    \frac{j}{2}\int_{\mathbb{S}^d}\big((j-1)f^2-|\nabla f|^2\big)\,d\sigma_{\mathbb{S}^d}.
\]
Applying the same calculation to the restriction $f|_{S_F}$ gives
\eqref{eq:projection-expansion}. We now show the uniformity of the
remainder.

Let
\[
    \mathcal I_j=\{(F,u):F\in\G_{n,j},\ u\in S_F\}
\]
be the incidence bundle. It is compact. For $(F,u)\in\mathcal I_j$, put
\[
    f_F=f|_{S_F},\qquad
    A_F(u)=\nabla^2_{F}f_F(u)+f_F(u)g_F .
\]
Since $f\in C^2(\mathbb{S}^{n-1})$, the quantities $|f_F(u)|$ and $\|A_F(u)\|$ are
bounded on $\mathcal I_j$ by a constant $M_j$. This follows either by writing
the intrinsic covariant derivatives on $S_F$ in local frames or, invariantly, by
noting that the restriction map $(F,u)\mapsto f|_{S_F}$ depends continuously on
$(F,u)$ up to second derivatives. Hence, for $|t|\le (2M_j+2)^{-1}$, all
matrices $I+tA_F(u)$ have eigenvalues in a fixed compact interval and all
functions $1+t f_F(u)$ are bounded away from zero.

For any symmetric $d\times d$ matrix $A$ with $\|A\|\le M_j$, Taylor expansion of
the polynomial $A\mapsto \det(I+tA)$ gives the quantitative bound
\[
\left|
\det(I+tA)-1-t\operatorname{tr}A
-\frac{t^2}{2}\big((\operatorname{tr}A)^2-\operatorname{tr}(A^2)\big)
\right|
\le C_{j,M_j}|t|^3 .
\]
The constant is uniform in $A$ because the coefficients of the determinant are
elementary symmetric polynomials of the eigenvalues and are bounded on
$\{\|A\|\le M_j\}$. Multiplying by $1+t f_F(u)$ only changes the constant, since
$|f_F(u)|\le M_j$. Thus the integrand in the support-volume formula has a
third-order remainder bounded by $C|t|^3$ for every $(F,u)\in\mathcal I_j$.
Integration over $S_F$, followed by division by $j\kappa_j$ (equivalently, integration with normalized measure), preserves the same bound, and so
\eqref{eq:projection-expansion} holds with an $O(t^3)$ term uniform in
$F\in\G_{n,j}$.
\end{proof}

We now pass from the expansion of projected volumes to the expansion of the normalized quantities $I_{j,-n}$.

We shall also use the uniformity of the remainder through the elementary operations in the next proposition. In the setting of that proposition the perturbing harmonic has mean zero; hence Lemma \ref{lem:projection-expansion} gives $V_j(F,t)=1+O(t)$ uniformly in $F$, while the global volume satisfies $V_n(t)=1+O(t^2)$. After shrinking $\varepsilon$ we have
\[
        \frac12\le V_j(F,t)\le2,
        \qquad
        \frac12\le V_n(t)\le2
\]
for every $F\in\G_{n,j}$ and $|t|\le\varepsilon$. On the interval $[1/2,2]$, the third derivatives of $x\mapsto x^\alpha$, $x\mapsto x^{-n}$, and $x\mapsto \log x$ are bounded for every fixed exponent $\alpha$. Taylor's theorem therefore gives uniform third-order remainders when we form $V_n(t)^{-j/n}$, multiply by $V_j(F,t)$, raise the quotient to the power $-n$, integrate over $\G_{n,j}$, and finally take the logarithm. Thus every $O(t^3)$ term in Proposition \ref{prop:second-variation} is uniform in the precise sense fixed above.

\begin{proposition}\label{prop:second-variation}
Let $Y$ be a nonconstant spherical harmonic on $\mathbb{S}^{n-1}$ satisfying
\[
        -\Delta_{\mathbb{S}^{n-1}}Y=\lambda Y,
        \qquad \int_{\mathbb{S}^{n-1}}Y\,d\sigma=0.
\]
Let $K_t$ have support function $1+tY$. Then, for every $1\le j\le n$,
\begin{equation}\label{eq:Aj-formula}
    \log I_{j,-n}(K_t)^{1/j}
    =A_j(Y)t^2+O(t^3),
\end{equation}
where
\begin{equation}\label{eq:Aj}
    A_j(Y)
    =
    \frac{n-j}{2}\left(\frac{\lambda}{n-1}-1\right)\|Y\|_2^2
    -\frac{n+1}{2j}\int_{\G_{n,j}}\big(jR_jY(F)\big)^2\,dF.
\end{equation}
\end{proposition}

\begin{proof}
Set
\[
    V_j(F,t)=\frac{\vol_j(P_FK_t)}{\kappa_j},
    \qquad
    V_n(t)=\frac{\vol_n(K_t)}{\kappa_n}.
\]
Since $B_{K_t}$ has radius $V_n(t)^{1/n}$, where $B_{K_t}$ denotes the centered Euclidean ball with the same volume as $K_t$, we have
\begin{equation}\label{eq:vj-def}
    I_{j,-n}(K_t)^{1/j}
    =
    \left(
       \int_{\G_{n,j}}
       \left(\frac{V_j(F,t)}{V_n(t)^{j/n}}\right)^{-n}dF
    \right)^{-1/(nj)}.
\end{equation}
By Lemma \ref{lem:projection-expansion},
\[
    V_j(F,t)=1+a_j(F)t+b_j(F)t^2+O(t^3),
\]
where
\[
    a_j(F)=jR_jY(F),
    \qquad
    b_j(F)=\frac{j}{2}R_j\big((j-1)Y^2-|\nabla_FY|^2\big)(F).
\]
For $j=n$, since $\int_{\mathbb{S}^{n-1}}Y\,d\sigma=0$ and $\int_{\mathbb{S}^{n-1}} |\nabla Y|^2 d\sigma=\lambda\|Y\|_2^2$,
\begin{equation}\label{eq:global-volume}
    V_n(t)=1+\frac{n}{2}\big((n-1)-\lambda\big)\|Y\|_2^2t^2+O(t^3).
\end{equation}
Consequently
\[
    \frac{V_j(F,t)}{V_n(t)^{j/n}}
    =1+a_j(F)t+c_j(F)t^2+O(t^3),
\]
with
\[
    c_j(F)=b_j(F)+\frac{j}{2}\big(\lambda-(n-1)\big)\|Y\|_2^2.
\]

The $O(t^3)$ in this normalized expansion is also uniform in $F$. Indeed,
Lemma \ref{lem:projection-expansion} gives
\[
   V_j(F,t)=1+a_j(F)t+b_j(F)t^2+t^3E_j(F,t),
   \qquad |E_j(F,t)|\le C
\]
for all $F$ and $|t|\le t_0$. The same argument with $j=n$ gives
\[
   V_n(t)=1+d_nt^2+t^3E_n(t),
   \qquad
   d_n=\frac n2((n-1)-\lambda)\|Y\|_2^2,
   \qquad |E_n(t)|\le C.
\]
After shrinking $t_0$ we may assume $1/2\le V_n(t)\le3/2$. Since the scalar
function $s\mapsto s^{-j/n}$ has bounded first three derivatives on
$[1/2,3/2]$, Taylor's theorem yields
\[
   V_n(t)^{-j/n}=1-\frac jn d_nt^2+t^3\widetilde E_n(t),
   \qquad |\widetilde E_n(t)|\le C.
\]
Multiplying the last display by the uniform expansion of $V_j(F,t)$ gives a
remainder $t^3\widetilde E_j(F,t)$ with $|\widetilde E_j(F,t)|\le C$ uniformly
in $F$.

We shall need the Grassmannian average of $c_j$. The identities
\begin{equation}\label{eq:Radon-average-identities}
    \int_{\G_{n,j}}R_j(Y^2)(F)\,dF=\|Y\|_2^2,
    \qquad
    \int_{\G_{n,j}}R_j(|\nabla_FY|^2)(F)\,dF
    =\frac{j-1}{n-1}\lambda\|Y\|_2^2
\end{equation}
are as follows.  The first one is just Fubini on the incidence relation
\(\{(F,u):u\in S_F\}\), whose induced marginal on \(\mathbb S^{n-1}\) is
\(\sigma\).  For the second identity, disintegrate the same flag integral by
first fixing $u\in \mathbb{S}^{n-1}$ and then integrating over those
$F\in\G_{n,j}$ with $u\in F$. On this fiber,
$T_uS_F=F\cap u^\perp$ ranges over $\G_{u^\perp,j-1}$ with its normalized
rotation-invariant probability measure, where $\G_{u^\perp,j-1}$ denotes the
Grassmannian of $(j-1)$-dimensional subspaces of the Euclidean space
$u^\perp$. Hence, for each tangent vector $v\in u^\perp$,
\[
    \int_{\G_{u^\perp,j-1}} |P_Ev|^2\,dE=\frac{j-1}{n-1}|v|^2,
\]
with the convention that the left-hand side is zero when \(j=1\).  Taking
\(v=\nabla_{\mathbb{S}^{n-1}}Y(u)\) and integrating in \(u\) gives the displayed
identity. Hence
\begin{equation}\label{eq:int-cj}
    \int_{\G_{n,j}}c_j(F)\,dF
    =
    \frac{j(n-j)}{2}\left(\frac{\lambda}{n-1}-1\right)\|Y\|_2^2.
\end{equation}

Finally, write
\[
    W_j(F,t)=\frac{V_j(F,t)}{V_n(t)^{j/n}}
        =1+a_j(F)t+c_j(F)t^2+t^3S_j(F,t),
        \qquad |S_j(F,t)|\le C .
\]
The functions $a_j,c_j,S_j$ are uniformly bounded for $F\in\G_{n,j}$ and
$|t|\le t_0$. After one further reduction of $t_0$, we also have
$1/2\le W_j(F,t)\le3/2$. Applying Taylor's theorem to $s\mapsto s^{-n}$ on
$[1/2,3/2]$ gives
\begin{equation}\label{eq:uniform-negative-power}
\begin{aligned}
    W_j(F,t)^{-n}
    &=1-na_j(F)t+
      \left(-nc_j(F)+\frac{n(n+1)}2a_j(F)^2\right)t^2
      +t^3\Theta_j(F,t),\\
    |\Theta_j(F,t)|&\le C,
\end{aligned}
\end{equation}
again uniformly in $F$. Thus integration over $\G_{n,j}$ preserves the
third-order bound. Since $\int_{\G_{n,j}}a_j(F)dF=j\int Yd\sigma=0$, equations
\eqref{eq:vj-def}, \eqref{eq:int-cj}, and \eqref{eq:uniform-negative-power}
give
\begin{align*}
\log I_{j,-n}(K_t)^{1/j}
&= -\frac1{nj}\log\int_{\G_{n,j}}\left(\frac{V_j(F,t)}{V_n(t)^{j/n}}\right)^{-n}dF \\
&=\left[\frac1j\int_{\G_{n,j}}c_j(F)dF
-\frac{n+1}{2j}\int_{\G_{n,j}}a_j(F)^2dF\right]t^2+O(t^3).
\end{align*}
The last remainder is a one-variable remainder: there are constants $C,t_0>0$,
depending only on $n$, $j$, and $Y$, such that
\[
 \left|\log I_{j,-n}(K_t)^{1/j}-A_j(Y)t^2\right|\le C|t|^3,
 \qquad |t|\le t_0.
\]
This is \eqref{eq:Aj} with its uniform interpretation.
\end{proof}

\section{The fourth harmonic and its Grassmannian averages}
\label{sec:fourth-harmonic}

Let $Y$ be the fourth zonal harmonic defined in \eqref{eq:Y-def}. Equivalently, $Y$ is the restriction to $\mathbb{S}^{n-1}$ of
\begin{equation}\label{eq:harm-poly}
    H(x)=x_1^4-\frac{6}{n+4}x_1^2|x|^2
    +\frac{3}{(n+2)(n+4)}|x|^4.
\end{equation}
A direct calculation shows that $\Delta_{\R^n}H=0$. Since $H$ is homogeneous of degree four,
\begin{equation}\label{eq:lambda-fourth}
        -\Delta_{\mathbb{S}^{n-1}}Y=4(n+2)Y.
\end{equation}
In particular, $\int_{\mathbb{S}^{n-1}}Y\,d\sigma=0$.

\begin{lemma}\label{lem:averages}
For the function $Y$ in \eqref{eq:Y-def},
\begin{equation}\label{eq:normY}
    \|Y\|_2^2
    =
    \frac{24(n-1)(n+1)}{n(n+2)^2(n+4)^2(n+6)},
\end{equation}
and, for every $1\le j\le n$,
\begin{equation}\label{eq:j-average}
    \int_{\G_{n,j}}\big(jR_jY(F)\big)^2\,dF
    =
    \frac{3j(n-j)(n-j+2)}{(j+2)(n-1)(n+1)}\,\|Y\|_2^2.
\end{equation}
Equivalently,
\begin{equation}\label{eq:Bj-general}
    B_j:=
    \frac{1}{j\|Y\|_2^2}
    \int_{\G_{n,j}}\big(jR_jY(F)\big)^2\,dF
    =
    \frac{3(n-j)(n-j+2)}{(j+2)(n-1)(n+1)}.
\end{equation}
\end{lemma}

\begin{proof}
We first compute the $L^2$ norm. Since all sphere integrals are normalized and
$u_1=\langle u,e_1\rangle$, the standard coordinate-moment formula on the unit
sphere is
\begin{equation}
    \int_{\mathbb{S}^{n-1}} u_1^{2r}\,d\sigma(u)=\frac{(1/2)_r}{(n/2)_r},
    \qquad r=0,1,2,\ldots,
    \label{eq:sphere-moments}
\end{equation}
where $(a)_r=a(a+1)\cdots(a+r-1)$ is the rising factorial. Expanding
\eqref{eq:Y-def} and substituting \eqref{eq:sphere-moments} with
\(r=1,2,3,4\), we get
\[
\begin{aligned}
\|Y\|_2^2
&=\int u_1^8\,d\sigma-\frac{12}{n+4}\int u_1^6\,d\sigma
  +\left(\frac{36}{(n+4)^2}
    +\frac{6}{(n+2)(n+4)}\right)\int u_1^4\,d\sigma \\
&\qquad
  -\frac{36}{(n+2)(n+4)^2}\int u_1^2\,d\sigma
  +\frac{9}{(n+2)^2(n+4)^2},
\end{aligned}
\]
which simplifies to \eqref{eq:normY}.

It remains to compute the Grassmannian square average. For $F\in\G_{n,j}$ put
\[
        T_j(F)=|P_Fe_1|^2.
\]
For a fixed $F$, let $x\in S_F$ be the integration variable. Since only the
vector $P_Fe_1$ matters on $F$, the elementary coordinate moments on
$S_F\simeq S^{j-1}$ give
\[
        \int_{S_F}x_1^2\,d\sigma_F(x)=\frac{T_j(F)}{j},
        \qquad
        \int_{S_F}x_1^4\,d\sigma_F(x)=\frac{3T_j(F)^2}{j(j+2)}.
\]
Therefore
\begin{equation}\label{eq:jRjY-general}
    jR_jY(F)
    =
    \frac{3}{j+2}T_j(F)^2
    -\frac{6}{n+4}T_j(F)
    +\frac{3j}{(n+2)(n+4)}.
\end{equation}

We next compute the moments of $T_j(F)$ in an elementary way.  Let
$F_0=\operatorname{span}\{e_1,\ldots,e_j\}$.  By the rotation invariance of the
normalized measures, for every continuous function $\varphi$ on $[0,1]$,
\[
    \int_{\G_{n,j}}\varphi(T_j(F))\,dF
    =\int_{\mathbb{S}^{n-1}}\varphi(|P_{F_0}u|^2)\,d\sigma(u).
\]
Indeed, both sides are the same invariant average of a function of the squared
length of the projection of a fixed unit vector onto a moving $j$-plane.  Since
$|P_{F_0}u|^2=u_1^2+\cdots+u_j^2$, the standard coordinate-moment formula
\[
    \int_{\mathbb{S}^{n-1}}u_1^{2\alpha_1}\cdots u_j^{2\alpha_j}\,d\sigma(u)
    =\frac{(1/2)_{\alpha_1}\cdots(1/2)_{\alpha_j}}
           {(n/2)_{\alpha_1+\cdots+\alpha_j}}
\]
and the multinomial theorem give, for the only four moments needed below,
\begin{align}
    M_1(j)&:=\int_{\G_{n,j}}T_j(F)\,dF
        =\frac{j}{n},\notag\\
    M_2(j)&:=\int_{\G_{n,j}}T_j(F)^2\,dF
        =\frac{j(j+2)}{n(n+2)},\notag\\
    M_3(j)&:=\int_{\G_{n,j}}T_j(F)^3\,dF
        =\frac{j(j+2)(j+4)}{n(n+2)(n+4)},\notag\\
    M_4(j)&:=\int_{\G_{n,j}}T_j(F)^4\,dF
        =\frac{j(j+2)(j+4)(j+6)}{n(n+2)(n+4)(n+6)}.
        \label{eq:general-T-moments}
\end{align}
Squaring \eqref{eq:jRjY-general} and using \eqref{eq:general-T-moments} gives
\[
\begin{aligned}
\int_{\G_{n,j}}(jR_jY)^2\,dF
&=
\frac{9}{(j+2)^2}M_4(j)
-\frac{36}{(j+2)(n+4)}M_3(j)  \\
&\quad+
\left(\frac{36}{(n+4)^2}
+\frac{18j}{(j+2)(n+2)(n+4)}\right)M_2(j) \\
&\quad-
\frac{36j}{(n+2)(n+4)^2}M_1(j)
+\frac{9j^2}{(n+2)^2(n+4)^2}  \\
&=\frac{72j(n-j)(n-j+2)}
        {n(j+2)(n+2)^2(n+4)^2(n+6)}.
\end{aligned}
\]
Together with \eqref{eq:normY}, this is \eqref{eq:j-average}. When $j=n$, both
sides of \eqref{eq:j-average} are zero because $R_nY=\int_{\mathbb{S}^{n-1}}Y\,d\sigma=0$.
\end{proof}

\section{Proof of the counterexample}
\label{sec:proof}

\begin{proof}[Proof of Theorem \ref{thm:main}]
Fix integers $m,k,n$ with $1\le m<k\le n-1$ and $n>(m+2)(k+2)-2$, and let $Y$ be given by
\eqref{eq:Y-def}. The matrix
\[
        \nabla^2_{\mathbb{S}^{n-1}}h_t+h_tg
        =g+t\big(\nabla^2_{\mathbb{S}^{n-1}}Y+Yg\big)
\]
is positive definite for all sufficiently small $|t|$, because $\mathbb{S}^{n-1}$ is compact. After shrinking $|t|$ if necessary, also $h_t>0$ on $\mathbb{S}^{n-1}$. These conditions imply that $h_t=1+tY$ is the support function of a $C^2_+$ convex body $K_t$. Since $Y$ is even, $K_t$ is origin-symmetric.

Apply Proposition \ref{prop:second-variation} with $\lambda=4(n+2)$. Then
\[
        \frac{\lambda}{n-1}-1
        =\frac{4(n+2)}{n-1}-1
        =\frac{3(n+3)}{n-1}.
\]
Since $1\le m<k\le n-1$, using Lemma \ref{lem:averages}, for $1\le j\le n-1$ we may write
\begin{equation}\label{eq:Aj-general-fourth}
    A_j(Y)
    =
    \frac12\left[
        (n-j)\frac{3(n+3)}{n-1}
        -(n+1)B_j
    \right]\|Y\|_2^2,
\end{equation}
where $B_j$ is given by \eqref{eq:Bj-general}. Equivalently,
\begin{equation}\label{eq:Aj-simplified-fourth}
    A_j(Y)=
    \frac{3(n+4)(n-j)(j+1)}{2(j+2)(n-1)}\|Y\|_2^2.
\end{equation}
Hence
\begin{align*}
    \log \frac{I_{m,-n}(K_t)^{1/m}}{I_{k,-n}(K_t)^{1/k}}
    &=\big(A_m(Y)-A_k(Y)\big)t^2+O(t^3)\\
    &=\frac{3(n+4)(k-m)\big((m+2)(k+2)-n-2\big)}
         {2(n-1)(m+2)(k+2)}\|Y\|_2^2t^2+O(t^3).
\end{align*}
Since $n>(m+2)(k+2)-2$, the coefficient of $t^2$ is strictly negative. More explicitly, set
\[
        c_{m,k,n}=\frac{3(n+4)(k-m)\big(n+2-(m+2)(k+2)\big)}
         {2(n-1)(m+2)(k+2)}\|Y\|_2^2>0 .
\]
By the uniform remainder statement above, there exist constants $\varepsilon>0$ and $C>0$, depending on the fixed dimension and the fixed perturbation $Y$, such that the absolute value of the remaining term is at most $C|t|^3$ whenever $|t|\le\varepsilon$. Choosing
\[
        0<|t|<\min\left\{\varepsilon,\frac{c_{m,k,n}}{2C}\right\}
\]
gives
\[
    \log \frac{I_{m,-n}(K_t)^{1/m}}{I_{k,-n}(K_t)^{1/k}}
    \le -c_{m,k,n}t^2+C|t|^3<0.
\]
Consequently,
\[
        I_{m,-n}(K_t)^{1/m}<I_{k,-n}(K_t)^{1/k}.
\]
This gives the strict reverse inequality to the corresponding comparison in
the conjectural chain \eqref{eq:conj-chain} at \(p=-n\).
\end{proof}

\section{The three-dimensional endpoint chain}\label{sec:dim3-chain}

This section proves Theorem \ref{thm:dim3-chain}.  We keep the notation and
normalizations from Section \ref{pre}.  Thus measures on spheres and
Grassmannians are normalized unless the unnormalized surface element \(d\omega\)
is explicitly displayed.  In this section all polar bodies are taken with
respect to the origin, and if a body lies in a proper subspace then its polar is
taken inside that subspace.

For \(K\subset\R^3\), write
\[
        w_K(u)=\vol_1(P_{\R u}K),\qquad
        b_K(u)=\vol_2(P_{u^\perp}K),\qquad u\in\mathbb S^2.
\]
Let \(B_K\) be the centered Euclidean ball with \(\vol_3(B_K)=\vol_3(K)\), and
let \(r_K\) be its radius.  Thus
\[
        \vol_3(K)=\vol_3(B_K)=\frac{4\pi}{3}r_K^3.
\]
Since functions on \(\G_{3,1}\) and \(\G_{3,2}\) may be represented by even
functions on \(\mathbb S^2\), the definition \eqref{eq:I-def} gives
\begin{align}
 I_{1,-3}(K)
 &=\frac{1}{2r_K}
      \left(\int_{\mathbb S^2}w_K(u)^{-3}\,d\sigma(u)\right)^{-1/3},
      \label{eq:dim3-I1-width}
      \\
 I_{2,-3}(K)^{1/2}
 &=\frac{1}{\sqrt\pi\,r_K}
      \left(\int_{\mathbb S^2}b_K(u)^{-3}\,d\sigma(u)\right)^{-1/6}.
      \label{eq:dim3-I2-bright}
\end{align}
Put
\[
        A_K=\int_{\mathbb S^2}w_K(u)^{-3}\,d\sigma(u),\qquad
        B_K'=\int_{\mathbb S^2}b_K(u)^{-3}\,d\sigma(u).
\]
Then \eqref{eq:dim3-I1-width} and \eqref{eq:dim3-I2-bright} show that
\[
        I_{1,-3}(K)\ge I_{2,-3}(K)^{1/2}
\]
is equivalent to
\begin{equation}\label{eq:dim3-sigma-form}
        B_K'\ge \frac{64}{\pi^3}A_K^2.
\end{equation}
Indeed, after canceling \(r_K\), one raises the equivalent inequality
\((B_K')^{1/6}\ge 2\pi^{-1/2}A_K^{1/3}\) to the sixth power.

For a convex body \(Q\subset\R^3\) containing the origin in its interior, the
polar-coordinate formula gives
\begin{equation}\label{eq:dim3-polar-volume-normalized}
        \vol_3(Q^\circ)
        =\frac13\int_{\mathbb S^2}h_Q(u)^{-3}\,d\omega(u)
        =\frac{4\pi}{3}\int_{\mathbb S^2}h_Q(u)^{-3}\,d\sigma(u).
\end{equation}
Let
\[
        DK=K+(-K)
\]
be the difference body, and let \(\Pi K\) be the projection body, normalized by
\[
        h_{\Pi K}(u)=\vol_2(P_{u^\perp}K)=b_K(u),\qquad u\in\mathbb S^2.
\]
Since \(h_{DK}=w_K\), \eqref{eq:dim3-sigma-form} is equivalent to
\begin{equation}\label{eq:dim3-target-K}
        \vol_3((\Pi K)^\circ)
        \ge \frac{48}{\pi^4}\vol_3((DK)^\circ)^2.
\end{equation}

Set
\[
        L=\frac12DK.
\]
Then \(L\) is origin-symmetric and
\begin{equation}\label{eq:dim3-Lpolar}
        L^\circ=2(DK)^\circ,
        \qquad \vol_3(L^\circ)=8\vol_3((DK)^\circ).
\end{equation}
For fixed \(u\in\mathbb S^2\), let \(A=P_{u^\perp}K\subset u^\perp\).  Since
projection commutes with Minkowski addition,
\[
        P_{u^\perp}L=\frac12(A-A).
\]
The two-dimensional Brunn--Minkowski inequality gives
\[
        \vol_2(A-A)^{1/2}\ge \vol_2(A)^{1/2}+\vol_2(-A)^{1/2}
        =2\vol_2(A)^{1/2}.
\]
Consequently
\begin{equation}\label{eq:dim3-BM-projection}
        \vol_2(P_{u^\perp}L)=\frac14\vol_2(A-A)
        \ge \vol_2(A)=\vol_2(P_{u^\perp}K).
\end{equation}
Equivalently, \(h_{\Pi L}\ge h_{\Pi K}\), and hence \(\Pi K\subseteq \Pi L\).
Passing to polars reverses inclusion, so
\[
        (\Pi L)^\circ\subseteq(\Pi K)^\circ,
        \qquad \vol_3((\Pi K)^\circ)\ge \vol_3((\Pi L)^\circ).
\]
Thus \eqref{eq:dim3-target-K} follows from the following inequality for
origin-symmetric bodies:
\begin{equation}\label{eq:dim3-DP}
        \vol_3((\Pi L)^\circ)
        \ge \frac{3}{4\pi^4}\vol_3(L^\circ)^2,
        \qquad L=-L\subset\R^3.
\end{equation}
Indeed, \eqref{eq:dim3-Lpolar} gives
\[
        \frac{3}{4\pi^4}\vol_3(L^\circ)^2
        =\frac{48}{\pi^4}\vol_3((DK)^\circ)^2.
\]
It remains to prove \eqref{eq:dim3-DP}.

Let \(M=L^\circ\).  For \(u\in\mathbb S^2\), define the central section
\[
        A_u=M\cap u^\perp,
\]
and set
\[
        p(u)=\vol_2(A_u^\circ),
\]
where \(A_u^\circ\) is the polar body of \(A_u\) inside the plane \(u^\perp\).
We shall use the elementary polarity identity
\begin{equation}\label{eq:dim3-projection-section-polarity}
        (P_E L)^\circ_E=L^\circ\cap E
\end{equation}
for every linear subspace \(E\subset\R^3\).  Here the polar on the left is taken
inside \(E\).  Indeed, for \(y\in E\), one has \(h_{P_E L}(y)=h_L(y)\), and hence
\[
        (P_E L)^\circ_E
        =\{y\in E:h_L(y)\le 1\}=L^\circ\cap E.
\]
Taking \(E=u^\perp\), we get
\[
        P_{u^\perp}L=(L^\circ\cap u^\perp)^\circ=A_u^\circ.
\]
Thus
\[
        h_{\Pi L}(u)=\vol_2(P_{u^\perp}L)=\vol_2(A_u^\circ)=p(u),
\]
and by \eqref{eq:dim3-polar-volume-normalized},
\[
        \vol_3((\Pi L)^\circ)=\frac13\int_{\mathbb S^2}p(u)^{-3}\,d\omega(u).
\]
Since \(\vol_3(L^\circ)=\vol_3(M)\), inequality \eqref{eq:dim3-DP} is equivalent
to
\begin{equation}\label{eq:dim3-section-target}
        \int_{\mathbb S^2}p(u)^{-3}\,d\omega(u)
        \ge \frac{9}{4\pi^4}\vol_3(M)^2.
\end{equation}

\subsection{A sharp planar determinant estimate}

If \(A\subset\R^2\) is an origin-symmetric convex body, define
\begin{equation}\label{eq:dim3-D-def}
        D(A)=\int_A\int_A |\det(x,y)|\,dx\,dy.
\end{equation}

\begin{lemma}\label{lem:dim3-planar}
Let \(A\subset\R^2\) be an origin-symmetric convex body.  Then
\begin{equation}\label{eq:dim3-planar-estimate}
        D(A)\le \frac{8\pi^4}{9}\vol_2(A^\circ)^{-3}.
\end{equation}
Equality holds if and only if \(A\) is a centered ellipse.
\end{lemma}

\begin{proof}
Let \(\Gamma A\) denote the centroid body of \(A\), normalized by
\[
        h_{\Gamma A}(\xi)=\frac1{\vol_2(A)}\int_A |\langle x,\xi\rangle|\,dx,
        \qquad \xi\in\R^2.
\]
By the mixed-volume formula for zonoids
\cite[Theorem 5.3.2, in particular formula (5.82)]{Schneider}, if
\(h_Z(\xi)=\int |\langle x,\xi\rangle|\,d\mu(x)\), then
\[
        \vol_2(Z)=2\int\int |\det(x,y)|\,d\mu(x)d\mu(y).
\]
Applying this to \(\mu=\vol_2(A)^{-1}{\bf 1}_A(x)dx\) gives
\begin{equation}\label{eq:dim3-centroid-area}
        \vol_2(\Gamma A)=\frac{2}{\vol_2(A)^2}D(A),
        \qquad\text{or equivalently}\qquad
        D(A)=\frac{\vol_2(A)^2}{2}\vol_2(\Gamma A).
\end{equation}

We next prove the sharp polar centroid estimate
\begin{equation}\label{eq:dim3-polar-centroid-planar}
        \vol_2(\Gamma A)\,\vol_2(A^\circ)\le \frac{16}{9}.
\end{equation}
Put \(Q=A^\circ\).  Let \(d\omega_1\) be arclength measure on \(\mathbb S^1\), and
consider the measure
\begin{equation}\label{eq:dim3-C-surface-area-measure}
        d\mu_Q(v)=h_Q(v)^{-3}\,d\omega_1(v),\qquad v\in\mathbb S^1.
\end{equation}
Since \(Q\) is origin-symmetric, \(h_Q\) is even; hence \(\mu_Q\) is even and its
first moment is zero.  Its density is strictly positive, so it is not
concentrated on an antipodal pair.  By the planar Minkowski existence theorem
\cite{Schneider}, there is a planar convex body \(C\), unique up to translation,
whose surface area measure is
\begin{equation}\label{eq:dim3-C-surface-area-C}
        dS_C=d\mu_Q.
\end{equation}
The measure \(dS_C\) is even; by uniqueness in the Minkowski problem, \(C\) is
centrally symmetric up to translation.

Since \(Q=A^\circ\), we have \(\rho_A(v)=h_Q(v)^{-1}\).  Polar coordinates give,
for every \(\xi\in\R^2\),
\begin{align*}
        h_{\Gamma A}(\xi)
        &=\frac1{\vol_2(A)}\int_A |\langle x,\xi\rangle|\,dx  \\
        &=\frac1{3\vol_2(A)}\int_{\mathbb S^1}|\langle v,\xi\rangle|
          \rho_A(v)^3\,d\omega_1(v) \\
        &=\frac1{3\vol_2(A)}\int_{\mathbb S^1}|\langle v,\xi\rangle|
          h_Q(v)^{-3}\,d\omega_1(v).
\end{align*}
For the planar projection body,
\[
        h_{\Pi C}(\xi)=\frac12\int_{\mathbb S^1}|\langle v,\xi\rangle|\,dS_C(v).
\]
Using \eqref{eq:dim3-C-surface-area-measure} and
\eqref{eq:dim3-C-surface-area-C}, we obtain
\begin{equation}\label{eq:dim3-Gamma-PiC-factorization}
        \Gamma A=\frac{2}{3\vol_2(A)}\,\Pi C.
\end{equation}
Because \(C\) is centrally symmetric up to translation and projection bodies are
translation invariant, after translating \(C\) to its center we have
\[
        h_{\Pi C}(\xi)=\vol_1(P_{\xi^\perp}C)=2h_C(R\xi),
\]
where \(R\) is a rotation by \(\pi/2\) in the plane.  Equivalently,
\[
        \Pi C=2RC .
\]
In particular,
\begin{equation}\label{eq:dim3-Pi-planar-area}
        \vol_2(\Pi C)=4\vol_2(C).
\end{equation}

The mixed area of \(C\) and \(Q\) is
\begin{equation}\label{eq:dim3-mixed-area-CQ}
        V(C,Q)
        =\frac12\int_{\mathbb S^1}h_Q\,dS_C
        =\frac12\int_{\mathbb S^1}h_Q(v)^{-2}\,d\omega_1(v)
        =\vol_2(Q^\circ)
        =\vol_2(A).
\end{equation}
By the planar Minkowski inequality,
\[
        V(C,Q)^2\ge \vol_2(C)\vol_2(Q).
\]
Together with \eqref{eq:dim3-mixed-area-CQ}, this gives
\begin{equation}\label{eq:dim3-C-volume-bound}
        \vol_2(C)\le \frac{\vol_2(A)^2}{\vol_2(Q)}
        =\frac{\vol_2(A)^2}{\vol_2(A^\circ)}.
\end{equation}
Combining \eqref{eq:dim3-Gamma-PiC-factorization},
\eqref{eq:dim3-Pi-planar-area}, and \eqref{eq:dim3-C-volume-bound}, we get
\begin{equation}\label{eq:dim3-Gamma-bound}
        \vol_2(\Gamma A)
        =\left(\frac{2}{3\vol_2(A)}\right)^2\vol_2(\Pi C)
        =\frac{16}{9\vol_2(A)^2}\vol_2(C)
        \le \frac{16}{9\vol_2(A^\circ)}.
\end{equation}
This proves \eqref{eq:dim3-polar-centroid-planar}.

Substituting \eqref{eq:dim3-Gamma-bound} into \eqref{eq:dim3-centroid-area}
gives
\begin{equation}\label{eq:dim3-D-before-santalo}
        D(A)\le \frac{\vol_2(A)^2}{2}\cdot\frac{16}{9\vol_2(A^\circ)}
        =\frac{8}{9}\frac{\vol_2(A)^2}{\vol_2(A^\circ)}.
\end{equation}
Finally, the planar Blaschke--Santalo inequality for origin-symmetric bodies
\cite{Schneider} states that
\[
        \vol_2(A)\vol_2(A^\circ)\le \pi^2,
\]
with equality if and only if \(A\) is a centered ellipse.  Hence
\[
        \vol_2(A)^2\le \frac{\pi^4}{\vol_2(A^\circ)^2}.
\]
Inserting this into \eqref{eq:dim3-D-before-santalo} yields
\[
        D(A)\le \frac{8\pi^4}{9}\vol_2(A^\circ)^{-3}.
\]
This proves \eqref{eq:dim3-planar-estimate}.

If equality holds in \eqref{eq:dim3-planar-estimate}, then equality must hold
in the last use of the planar Blaschke--Santalo inequality, and therefore \(A\)
is a centered ellipse.  Conversely, the disk gives equality by the direct
computation
\[
        D(B_2^2)=\frac{8\pi}{9},\qquad \vol_2((B_2^2)^\circ)=\pi,
\]
and the inequality is invariant under non-singular linear maps in the plane:
\(D(TA)=|\det T|^3D(A)\) and
\(\vol_2((TA)^\circ)=|\det T|^{-1}\vol_2(A^\circ)\).  Hence every centered
ellipse gives equality.
\end{proof}

\subsection{The Blaschke--Petkantschin formula}

We use the following form of the Blaschke--Petkantschin formula
\cite[Theorem 7.2.1]{SchneiderWeil} in \(\R^3\): for every non-negative measurable function
\(F\) on \(\R^3\times\R^3\),
\begin{equation}\label{eq:dim3-BP}
\int_{\mathbb S^2}\int_{u^\perp}\int_{u^\perp}
        F(x,y)|\det_{u^\perp}(x,y)|\,dx\,dy\,d\omega(u)
 =2\int_{\R^3}\int_{\R^3}F(x,y)\,dx\,dy.
\end{equation}
Here \(|\det_{u^\perp}(x,y)|\) denotes the Euclidean area of the parallelogram
spanned by \(x,y\) inside the plane \(u^\perp\).  The factor \(2\) corresponds to
the two unit normals to a two-dimensional subspace.

Apply \eqref{eq:dim3-BP} to \(F(x,y)={\bf 1}_M(x){\bf 1}_M(y)\).  Since
\(A_u=M\cap u^\perp\), the left-hand side becomes
\(\int_{\mathbb S^2}D(A_u)\,d\omega(u)\).  Thus
\begin{equation}\label{eq:dim3-int-D-sections}
        \int_{\mathbb S^2}D(A_u)\,d\omega(u)=2\vol_3(M)^2.
\end{equation}
By Lemma \ref{lem:dim3-planar}, applied in the plane \(u^\perp\) to the
origin-symmetric body \(A_u\),
\begin{equation}\label{eq:dim3-section-D-bound}
        D(A_u)\le \frac{8\pi^4}{9}\vol_2(A_u^\circ)^{-3}
        =\frac{8\pi^4}{9}p(u)^{-3}.
\end{equation}
Equivalently,
\[
        p(u)^{-3}\ge \frac{9}{8\pi^4}D(A_u).
\]
Integrating this inequality and using \eqref{eq:dim3-int-D-sections}, we obtain
\[
\int_{\mathbb S^2}p(u)^{-3}\,d\omega(u)
\ge \frac{9}{8\pi^4}\int_{\mathbb S^2}D(A_u)\,d\omega(u)
=\frac{9}{4\pi^4}\vol_3(M)^2.
\]
This is \eqref{eq:dim3-section-target}, and hence \eqref{eq:dim3-DP} follows.
Consequently \eqref{eq:dim3-target-K} holds for every convex body
\(K\subset\R^3\), and therefore
\begin{equation}\label{eq:dim3-first-comparison-proved}
        I_{1,-3}(K)\ge I_{2,-3}(K)^{1/2}.
\end{equation}

\subsection{Equality cases}

We now prove the equality statements.  Translations do not change projection
volumes, widths, brightness functions, or the quantities \(I_{j,-3}\).  If
\(T\in GL(3)\), then
\[
        D(TK)=T(DK),
        \qquad
        \Pi(TK)=|\det T|\,T^{-t}\Pi K.
\]
Consequently
\[
        \vol_3((D(TK))^\circ)=|\det T|^{-1}\vol_3((DK)^\circ),
        \qquad
        \vol_3((\Pi(TK))^\circ)=|\det T|^{-2}\vol_3((\Pi K)^\circ),
\]
while \(r_{TK}=|\det T|^{1/3}r_K\).  The polar formulas
\eqref{eq:dim3-I1-width}--\eqref{eq:dim3-I2-bright} therefore show that
\(I_{1,-3}\), \(I_{2,-3}\), and \(I_{3,-3}\) are invariant under non-singular
affine transformations.  Since balls give equality directly, all ellipsoids give
equality throughout the chain.

Conversely, suppose first that equality holds in
\[
        I_{1,-3}(K)=I_{2,-3}(K)^{1/2}.
\]
Then equality holds in \eqref{eq:dim3-target-K}.  In the proof of
\eqref{eq:dim3-target-K} we used the two inequalities
\[
        \vol_3((\Pi K)^\circ)\ge \vol_3((\Pi L)^\circ)
        \ge \frac{3}{4\pi^4}\vol_3(L^\circ)^2,
        \qquad L=\frac12DK.
\]
Since the final constants agree exactly with \eqref{eq:dim3-target-K}, equality
in \eqref{eq:dim3-target-K} forces equality in both inequalities above.  In
particular, \((\Pi L)^\circ\subseteq(\Pi K)^\circ\) and the two bodies have the
same volume; hence \((\Pi L)^\circ=(\Pi K)^\circ\), and therefore
\(\Pi L=\Pi K\).  Thus
\[
        \vol_2(P_{u^\perp}L)=\vol_2(P_{u^\perp}K)
        \qquad\text{for every }u\in\mathbb S^2.
\]
Comparing this with \eqref{eq:dim3-BM-projection}, equality holds in the
Brunn--Minkowski inequality for the planar bodies \(A=P_{u^\perp}K\) and \(-A\)
for every \(u\).  The equality condition in the Brunn--Minkowski inequality
implies that \(A\) and \(-A\) are translates of one another.  Hence every
two-dimensional orthogonal projection of \(K\) is centrally symmetric.

We spell out why this forces \(K\) itself to be centrally symmetric up to
translation.  Define the odd, positively homogeneous function
\[
        f(x)=h_K(x)-h_K(-x),\qquad x\in\R^3.
\]
If \(E\) is a two-dimensional subspace and \(P_EK\) is centrally symmetric with
center \(c_E\in E\), then, for every \(x\in E\),
\[
        h_K(x)-h_K(-x)
        =h_{P_EK}(x)-h_{P_EK}(-x)
        =2\langle c_E,x\rangle.
\]
Thus \(f|_E\) is linear on every two-dimensional subspace \(E\).  Since any two
vectors in \(\R^3\) lie in such a subspace, \(f\) is additive; together with its
homogeneity, this shows that \(f\) is a linear functional.  Hence there is
\(a\in\R^3\) such that \(f(x)=2\langle a,x\rangle\) for all \(x\).  Therefore
\[
        h_{K-a}(x)=h_K(x)-\langle a,x\rangle
        =h_K(-x)+\langle a,x\rangle=h_{K-a}(-x),
\]
so \(K-a\) is origin-symmetric.  After translating \(K\), we may therefore assume
\(K=-K\).  In this position, \(DK=2K\), so \(L=K\).

Equality also holds in \eqref{eq:dim3-DP}.  In the proof of
\eqref{eq:dim3-DP}, the only pointwise inequality was the planar determinant
estimate \eqref{eq:dim3-section-D-bound}.  Since the difference between the
right-hand side and the left-hand side of \eqref{eq:dim3-section-D-bound} is
non-negative and its integral is zero, equality holds in Lemma
\ref{lem:dim3-planar} for \(d\omega\)-almost every central section
\[
        A_u=K^\circ\cap u^\perp.
\]
Therefore \(K^\circ\cap u^\perp\) is an ellipse for almost every
\(u\in\mathbb S^2\).  The full-measure set of such directions is dense in
\(\mathbb S^2\).  Moreover, the set of directions for which the central section
is an ellipse is closed.  Indeed, if \(u_i\to u\), choose rotations
\(R_i:u^\perp\to u_i^\perp\) with \(R_i\to {\rm Id}\).  The radial functions of
\(R_i^{-1}(K^\circ\cap u_i^\perp)\) then converge uniformly on
\(u^\perp\cap\mathbb S^2\) to the radial function of \(K^\circ\cap u^\perp\).
Since \(K^\circ\) contains some Euclidean ball \(rB_2^3\) and is contained in
some \(RB_2^3\), these planar sections are uniformly non-degenerate.  A centered
ellipse in a fixed plane is the unit ball of a positive definite quadratic form;
the uniform bounds give compactness of the corresponding matrices.  Hence the
Hausdorff limit of these centered ellipses is again a centered ellipse.
Therefore every central plane section of \(K^\circ\) is an ellipse.

By Busemann's ellipsoid characterization theorem \cite{Busemann1949,Gardner}, a
three-dimensional convex body whose central plane sections through a fixed
interior point are ellipses is an ellipsoid.  Thus \(K^\circ\), and hence \(K\),
is an ellipsoid.

\section{The first endpoint comparison in dimension four}
\label{dim4}
The same polar-section argument also gives the first endpoint comparison in
dimension four.  We record the argument, since it is formally parallel to the
proof of \eqref{eq:dim3-first-comparison-proved}.

Let \(K\subset\R^4\) be a convex body, and let \(B_K\) be the centered Euclidean
ball with \(\vol_4(B_K)=\vol_4(K)\).  Write \(r_K\) for its radius, so that
\[
        \vol_4(K)=\kappa_4 r_K^4,
        \qquad
        \kappa_4=\frac{\pi^2}{2}.
\]
For \(u\in\mathbb S^3\) and \(F\in\G_{4,2}\), put
\[
        w_K(u)=\vol_1(P_{\R u}K),
        \qquad
        b_K(F)=\vol_2(P_FK),
\]
and set
\[
        A_K=\int_{\mathbb S^3}w_K(u)^{-4}\,d\sigma(u),
        \qquad
        B_K^{(2)}=\int_{\G_{4,2}}b_K(F)^{-4}\,dF .
\]
Here the integral over \(\mathbb S^3\) agrees with the integral over
\(\G_{4,1}\), since \(w_K\) is even and the normalized measure on
\(\G_{4,1}\) is the push-forward of the normalized measure on \(\mathbb S^3\)
under the antipodal quotient map.

By the definition \eqref{eq:I-def},
\[
        I_{1,-4}(K)
        =
        \frac{1}{2r_K}A_K^{-1/4},
        \qquad
        I_{2,-4}(K)^{1/2}
        =
        \frac{1}{\sqrt{\pi}\,r_K}\big(B_K^{(2)}\big)^{-1/8}.
\]
Hence
\[
        I_{1,-4}(K)\ge I_{2,-4}(K)^{1/2}
\]
is equivalent to
\begin{equation}\label{eq:dim4-first-target}
        B_K^{(2)}\ge \frac{256}{\pi^4}A_K^2 .
\end{equation}

We shall use the following sharp planar estimate.  If \(A\subset\R^2\) is an
origin-symmetric convex body and
\[
        D_2(A)=\int_A\int_A |\det(x,y)|^2\,dx\,dy,
\]
then
\begin{equation}\label{eq:dim4-planar-D2}
        D_2(A)\le \frac{\pi^6}{8}\vol_2(A^\circ)^{-4}.
\end{equation}
Indeed, if
\[
        M_A=\int_A xx^t\,dx ,
\]
then a direct computation gives
\[
        D_2(A)=2\det M_A .
\]
By Ball's quadratic Santal\'o inequality, recently proved in
\cite{BoroczkyPatsalosSaroglou},
\[
        \operatorname{tr}(M_AM_{A^\circ})
        =
        \int_A\int_{A^\circ}\langle x,y\rangle^2\,dx\,dy
        \le
        \int_{B_2^2}\int_{B_2^2}\langle x,y\rangle^2\,dx\,dy
        =
        \frac{\pi^2}{8}.
\]
On the other hand,
\[
        \operatorname{tr}(M_AM_{A^\circ})
        \ge 2\sqrt{\det M_A\,\det M_{A^\circ}} .
\]
We also use the planar moment lower bound
\[
        \det M_C\ge \frac{\vol_2(C)^4}{16\pi^2},
\]
with equality for centered ellipses.  Applying this to \(C=A^\circ\), and
writing \(V=\vol_2(A^\circ)\), we get
\[
        2\sqrt{\det M_A}\,\frac{V^2}{4\pi}
        \le
        \frac{\pi^2}{8}.
\]
Thus
\[
        \det M_A\le \frac{\pi^6}{16}V^{-4}.
\]
Since \(D_2(A)=2\det M_A\), \eqref{eq:dim4-planar-D2} follows.  Equality holds
for centered ellipses.

We now prove \eqref{eq:dim4-first-target}.  Let
\[
        DK=K+(-K),
        \qquad
        L=\frac12DK,
\]
and put
\[
        M=L^\circ .
\]
Then \(L\) is origin-symmetric and
\[
        w_K(u)=h_{DK}(u)=2h_L(u).
\]
Since \(M=L^\circ\), the polar-coordinate formula in \(\R^4\), with normalized
spherical measure, gives
\[
        \vol_4(M)
        =
        \kappa_4\int_{\mathbb S^3}h_L(u)^{-4}\,d\sigma(u).
\]
Consequently,
\begin{equation}\label{eq:dim4-AK-M}
        A_K
        =
        2^{-4}\int_{\mathbb S^3}h_L(u)^{-4}\,d\sigma(u)
        =
        \frac{\vol_4(M)}{16\kappa_4}.
\end{equation}

For every \(F\in\G_{4,2}\), projection commutes with Minkowski addition, and
therefore
\[
        P_FL=\frac12(P_FK-P_FK).
\]
The two-dimensional Brunn--Minkowski inequality gives
\[
        \vol_2(P_FL)\ge \vol_2(P_FK).
\]
Hence, since the exponent is negative,
\begin{equation}\label{eq:dim4-BK-to-L}
        B_K^{(2)}
        \ge
        \int_{\G_{4,2}}\vol_2(P_FL)^{-4}\,dF .
\end{equation}

We next use projection-section polarity.  For \(y\in F\),
\[
        h_{P_FL}(y)=h_L(y).
\]
Thus, taking the polar inside \(F\),
\[
        (P_FL)^\circ_F
        =
        \{y\in F:h_{P_FL}(y)\le 1\}
        =
        \{y\in F:h_L(y)\le 1\}
        =
        L^\circ\cap F
        =
        M\cap F.
\]
Therefore
\[
        \vol_2(P_FL)=\vol_2((M\cap F)^\circ),
\]
where the polar is taken inside \(F\).  From \eqref{eq:dim4-BK-to-L} and
\eqref{eq:dim4-planar-D2}, applied to \(A=M\cap F\), we get
\begin{equation}\label{eq:dim4-BK-D2}
\begin{aligned}
        B_K^{(2)}
        &\ge
        \int_{\G_{4,2}}\vol_2((M\cap F)^\circ)^{-4}\,dF     \\
        &\ge
        \frac{8}{\pi^6}
        \int_{\G_{4,2}}D_2(M\cap F)\,dF .
\end{aligned}
\end{equation}

It remains to use the Blaschke--Petkantschin formula in \(\R^4\).  With the
normalized Haar probability measure on \(\G_{4,2}\), it gives
\begin{equation}\label{eq:dim4-BP-D2}
        \int_{\G_{4,2}}D_2(M\cap F)\,dF
        =
        \frac{1}{2\pi^2}\vol_4(M)^2 .
\end{equation}
The constant is obtained by applying the identity to \(M=B_2^4\), for which
\(M\cap F=B_2^2\) for every \(F\), \(D_2(B_2^2)=\pi^2/8\), and
\(\vol_4(B_2^4)=\kappa_4=\pi^2/2\).

Combining \eqref{eq:dim4-BK-D2} and \eqref{eq:dim4-BP-D2}, we obtain
\[
        B_K^{(2)}
        \ge
        \frac{8}{\pi^6}\cdot\frac{1}{2\pi^2}\vol_4(M)^2
        =
        \frac{4}{\pi^8}\vol_4(M)^2 .
\]
On the other hand, by \eqref{eq:dim4-AK-M} and \(\kappa_4=\pi^2/2\),
\[
        \frac{256}{\pi^4}A_K^2
        =
        \frac{256}{\pi^4}\cdot
        \frac{\vol_4(M)^2}{16^2\kappa_4^2}
        =
        \frac{4}{\pi^8}\vol_4(M)^2 .
\]
Therefore \eqref{eq:dim4-first-target} holds, and hence
\[
        I_{1,-4}(K)\ge I_{2,-4}(K)^{1/2}.
\]

\end{document}